\documentclass[11pt,reqno]{amsart}
\usepackage[utf8]{inputenc}
\usepackage[T1]{fontenc}
\usepackage[english]{babel}
\usepackage[usenames,dvipsnames]{xcolor}
\definecolor{darkred}{rgb}{0.7,0,0}
\definecolor{darkblue}{rgb}{0,0,0.7}
\usepackage[colorlinks=true,urlcolor=darkblue,linkcolor=darkred,citecolor=darkblue]{hyperref}
\usepackage{amsfonts}
\usepackage{amssymb}
\usepackage{amsmath}
\usepackage{stmaryrd}
\usepackage{dsfont}
\usepackage{amsthm}
\usepackage[left=3cm, right=3cm, top=2.2cm,bottom=2.2cm]{geometry}
\usepackage{parskip}
\usepackage{mathtools}
\usepackage{tikz}
\usepackage{tikz-cd}

\patchcmd{\section}{\scshape}{\bfseries}{}{}
\makeatletter
\renewcommand{\@secnumfont}{\bfseries}
\makeatother

\newtheoremstyle{standard}{9pt}{9pt}{\itshape}{}{\bfseries}{.}{.5em}{}
\theoremstyle{standard}

\newtheorem{lemma}{Lemma}[section]
\newtheorem{thm}[lemma]{Theorem}
\newtheorem{prop}[lemma]{Proposition}

\newtheoremstyle{definition}{9pt}{9pt}{}{}{\bfseries}{.}{.5em}{}
\theoremstyle{definition}
\newtheorem{defi}[lemma]{Definition}
\newtheorem{rem}[lemma]{Remark}

\usepackage{cleveref}
\crefname{lemma}{Lemma}{Lemmas}
\crefname{thm}{Theorem}{Theorems}
\crefname{prop}{Proposition}{Propositions}
\crefname{cor}{Corollary}{Corollaries}
\crefname{thmintro}{Theorem}{Theorems}
\crefname{propintro}{Proposition}{Propositions}
\crefname{defi}{Definition}{Definitions}
\crefname{rem}{Remark}{Remarks}
\crefname{ex}{Example}{Examples}
\crefname{section}{Section}{Sections}

\AddToHook{env/lemma/begin}{\crefalias{lemma}{lemma}}
\AddToHook{env/thm/begin}{\crefalias{lemma}{thm}}
\AddToHook{env/prop/begin}{\crefalias{lemma}{prop}}
\AddToHook{env/cor/begin}{\crefalias{lemma}{cor}}
\AddToHook{env/defi/begin}{\crefalias{lemma}{defi}}
\AddToHook{env/rem/begin}{\crefalias{lemma}{rem}}
\AddToHook{env/ex/begin}{\crefalias{lemma}{ex}}

\DeclareMathOperator{\Hom}{Hom}
\DeclareMathOperator{\Aut}{Aut}
\DeclareMathOperator{\im}{im}
\DeclareMathOperator{\ord}{ord}
\DeclareMathOperator{\ev}{ev}
\newcommand{\alg}[1]{\overline{#1}}

\begin{document}

% ----------------------- ABSTRACT ----------------------- 

\begin{abstract}
We present a simple proof of the fundamental theorem of Galois theory, which establishes a correspondence between the intermediate fields of a finite Galois extension and the subgroups of its Galois group. The proof is based on the combinatorial fact that a field cannot be expressed as the union of finitely many proper subfields.
\end{abstract}

% ----------------------- TITLE ----------------------- 

\title[A simple proof of the fundamental theorem of Galois theory]{A simple proof of the fundamental\\theorem of Galois theory}
\author{Martin Brandenburg}
\thanks{\emph{E-mail:} \texttt{brandenburg@uni-muenster.de}}
\date{\today}

\maketitle

% ----------------------- INTRODUCTION ----------------------- 

\section{Introduction}

The fundamental theorem of finite Galois theory states that for every finite Galois, i.e.\ normal and separable extension $L/K$ the maps $E \mapsto \Aut_E(L)$ and $L^H \mapsfrom H$ establish a bijection
\[\bigl\{\text{intermediate fields of } L/K\bigr\} \cong \bigl\{\text{subgroups of } \Aut_K(L)\bigr\}.\]
This theorem can be separated into two parts:
\begin{enumerate}
\item For every intermediate field $E$ of $L/K$ the trivial inclusion ${E \subseteq L^{\Aut_E(L)}}$ is an equality.
\item For every subgroup $H$ of $\Aut_K(L)$ the trivial inclusion $H \subseteq \Aut_{L^H}(L)$ is an equality.
\end{enumerate}
There are several standard proofs which can be found in textbooks on algebra, for example \cite[Thm.\ 16]{Artin}, \cite[Thm.\ V.§33.4]{Bourbaki}, \cite[Thm.\ 14.14]{Dummit}, \cite[Thm.\ VII.6.9]{Aluffi}. This note is the result of an attempt to prove (1) and (2) as directly as possible from the definitions. In particular, we will not prove (1) and (2) by comparing the degrees resp.\ orders of both sides. We will not use the linear independence of characters either. Splitting fields are not even mentioned once.

Instead, (1) will be derived rather directly from the definitions and basic facts about algebraic extensions, whereas (2) will be derived rather easily from a combinatorial result, namely that a field cannot be written as a union of finitely many proper subfields. The same result has been used by Geck \cite{Geck} to give a very short proof of the well-known equivalent characterizations of Galois extensions, which in turn leads to a short proof of the fundamental theorem. Other quick proofs have been found by DeMeyer \cite{DeMeyer} and Dress \cite{Dress}. Dress's approach is very conceptual as it uses general facts about group actions on sets and vector spaces.
 
This note does not assume prior knowledge about Galois theory. We only assume basic field theory and develop in detail all the ingredients here which are necessary to state and prove the fundamental theorem of Galois theory. As such, this note can also be used as an introduction to Galois theory. We will, however, omit most results which are not necessary for the theorem.

\emph{Acknowledgments.} I would like to thank Peter Müller for pointing me to Geck's paper \cite{Geck}.

% ----------------------- PRELIMINARIES ----------------------- 

\section{Preliminaries}

For basic field theory we refer to \cite[sect.\ VII.1, VI.2.1]{Aluffi}. A field extension is, by (modern) definition, a homomorphism of fields. It is automatically injective. We assume the notions of finite and algebraic extensions and their homomorphisms, the degree of a field extension, the multiplicativity formula for degrees, minimal polynomials as well as algebraic closures. For a field extension $L/K$ and an algebraic element $a \in L$ we have $K(a) = K[a] \cong K[T] / \langle f \rangle$, where $f \in K[T]$ is the minimal polynomial of $a$ over $K$. This is all we need from field theory.

For extensions $L/K$, $M/K$ we denote by $\Hom_K(L,M)$ the set of $K$-homomorphisms $L \to M$. We denote by $\Aut_K(L)$ the group of $K$-automorphisms of $L$.

\begin{lemma} \label{univ}
Let $L/K$, $M/K$ be extensions. Let $a \in L$ be algebraic over $K$ with minimal polynomial $f \in K[T]$. Then we have a bijection
\[\Hom_K(K(a),M) \cong \{m \in M : f(m) = 0\}\]
given by $\sigma \mapsto \sigma(a)$. In particular, $\Hom_K(K(a),M)$ has at most $\deg(f) = [K(a):K]$ elements.
\end{lemma}

\begin{proof}
This follows from $K(a) \cong K[T]/\langle f \rangle$ as well as the universal properties of quotient algebras and polynomial algebras:
\begin{align*}
\Hom_K(K(a),M) & \cong \Hom_K(K[T]/\langle f \rangle,M) \\
& \cong \{\sigma \in \Hom_K(K[T],M) : \sigma(f) = 0\} \\
& \cong \{m \in M : \ev_m(f) = 0\} \quad (\text{where } \ev_m : K[T] \to M,\, T \mapsto m) \\
& = \{m \in M : f(m) = 0\} \qedhere
\end{align*}
\end{proof}

\begin{lemma} \label{atmost}
Let $L/K$ be a finite extension. Let $M/K$ be any extension. Then $\Hom_K(L,M)$ has at most $[L:K]$ elements.
\end{lemma}

\begin{proof}
We need to show that the given homomorphism $K \to M$ admits at most $[L:K]$ extensions to $L$. We proceed by induction on $[L:K]$. The case $[L:K]=1$ is easy. Now assume $[L:K] > 1$ and pick some $a \in L \setminus K$. Because of \cref{univ} there are at most $[K(a):K]$ extensions of $K \to M$ to $K(a)$. We have $[L:K(a)] < [L:K]$. By induction hypothesis, each extension $K(a) \to M$ has at most $[L:K(a)]$ extensions to $L$. Therefore, there are at most $[L:K(a)] \cdot [K(a):K] = [L:K]$ extensions of $K \to M$ to $L$.
\end{proof}

\begin{lemma}\label{extend}
Let $L/K$ be an algebraic extension. Let $M$ be an algebraically closed field. Then every homomorphism $\sigma : K \to M$ admits an extension $\tau : L \to M$.
\end{lemma}

\begin{proof}
Consider the set of pairs $(E,\tau)$, where $K \subseteq E \subseteq L$ is an intermediate field and $\tau : E \to M$ is a homomorphism extending $\sigma$. We define $(E,\tau) \leq (E',\tau')$ by $E \subseteq E'$ and $\tau'|_E = \tau$. This defines a partial order in which every chain has an upper bound -- simply take the union. Thus, by Zorn's Lemma there is a maximal pair $(E,\tau)$, and we need to prove $E = L$. Let $a \in L$ with minimal polynomial $f \in E[T]$ over $E$. By means of $\tau : E \to M$ we can see $M$ as an extension of $E$. Then the image of $f$ in $M[T]$ has a root in $M$, since $M$ is algebraically closed. By \cref{univ} we can therefore extend $\tau$ to a homomorphism $E(a) \to M$. Since $(E,\tau)$ is maximal, this shows $E(a) = E$, so that $a \in E$.
\end{proof}

\begin{rem}
In the case of a finite extension, which we are mainly interested in, Zorn's Lemma is not necessary to prove \cref{extend}. Here, an induction on the degree does the job.
\end{rem}
 
\begin{lemma} \label{auto}
Let $L/K$ be an algebraic extension. Then
\[\Hom_K(L,L) = \Aut_K(L).\]
\end{lemma}

\begin{proof}
Let $\sigma : L \to L$ be a $K$-homomorphism. Of course, $\sigma$ is injective. In order to show that $\sigma$ is surjective, let $a \in L$ and let $f \in K[T]$ be its minimal polynomial. Let $N \subseteq L$ be the set of roots of $f$ in $L$. Then $\sigma(N) \subseteq N$, so that $\sigma$ restricts to an injective map $N \to N$. Since $N$ is finite, it has to be surjective. In particular, $a \in N$ has a preimage.
\end{proof}

% ----------------------- COMBINATORIAL RESULTS ----------------------- 

\section{Combinatorial results}

It is a well-known fact that a vector space cannot be the union of two proper subspaces. The following combinatorial results are variants of this fact and are contained in \cite{BBS}. We include the proofs for the convenience of the reader.
 
\begin{lemma} \label{vecunion}
A vector space over an infinite field cannot be written as the union of finitely many proper subspaces.
\end{lemma}

\begin{proof}
Let $K$ be an infinite field and $V$ be a vector space over $K$ which can be written as
\[V = V_1 \cup \cdots \cup V_n\]
with proper subspaces $V_1,\dotsc,V_n$. We use induction on $n$. The cases $n=0,1$ are trivial. Let's assume $n \geq 2$. By induction hypothesis there is some $v \in V \setminus (V_2 \cup \cdots \cup V_n)$. Then $v \in V_1$. Choose some $w \in V \setminus V_1$. For every $\lambda \in K^{\times}$ we have $v + \lambda w \notin V_1$, and these are infinitely many vectors. Thus, there is some $V_j$ with $1 < j \leq n$ which contains infinitely many of these vectors. Subtracting two of them yields $(v + \lambda w) - (v +  \lambda' w) = (\lambda - \lambda')w$, so that $w \in V_j$ and thus $v = (v + \lambda w) - \lambda w \in V_j$, which is a contradiction.
\end{proof}

It is worth mentioning that \cref{vecunion} can be used to prove the primitive element theorem \cite[Thm.\ 5.4.11]{Douady}.

\begin{lemma} \label{groupintersect}
Let $G$ be an infinite group. Let $G_1,\dotsc,G_n$ be finitely many subgroups of $G$ with $G = \bigcup_{1 \leq i \leq n} G_i$ (as sets) such that $G \neq \bigcup_{1 \leq i \leq n,\, i \neq j} G_i$ for all $j$. Then their intersection $\bigcap_{1 \leq i \leq n} G_i$ is infinite.
\end{lemma}

\begin{proof}
By induction on $k$ we will prove that there are pairwise distinct indices $i_1,\dotsc,i_k$ such that $G_{i_1} \cap \cdots \cap G_{i_k}$ is infinite; the case $k=n$ then finishes the proof. As for the case  $k=1$, since $G$ is infinite at least one $G_i$ has to be infinite as well. Now let $ k < n$ and assume that the claim is proven for $k$. By assumption we have $G \neq G_{i_1} \cup \cdots \cup G_{i_k}$. Pick some $b \in G$ with $b \notin G_{i_1} \cup \cdots \cup G_{i_k}$. For every element $a$ of the infinite group $G_{i_1}  \cap \cdots \cap G_{i_k}$ we have $ab \notin  G_{i_1} \cup \cdots \cup G_{i_k}$, which yields an index $j \neq i_1,\dotsc,i_k$ with $ab \in G_j$. So there must be some index $i_{k+1} \neq i_1,\dotsc,i_k$ such that the set
\[S \coloneqq \bigl\{a \in G_{i_1}  \cap \cdots \cap G_{i_k} : ab \in G_{i_{k+1}}\bigr\}\]
is infinite. For all $a,a' \in S$ we then have $aa'^{-1} = (ab)(a'b)^{-1} \in G_{i_{k+1}}$, on the other hand also $aa'^{-1} \in G_{i_1} \cap \cdots \cap G_{i_k}$. Therefore $SS^{-1} \subseteq G_{i_1} \cap \cdots \cap G_{i_k} \cap G_{i_{k+1}}$ is infinite.
\end{proof}

\begin{lemma} \label{fieldunion}
A field cannot be written as the union of finitely many proper subfields.
\end{lemma}

\begin{proof}
Assume $L$ is a field such that $L = L_1 \cup \cdots \cup L_n$ with proper subfields $L_1,\dotsc,L_n$. If $L$ is finite, then $L^{\times}$ is cyclic. Choose a generator of $L^{\times}$. It lies in some $L_i$, so that $L = L_i$, contradiction. Now assume that $L$ is infinite. We proceed by induction on $n$. The case $n=0$ is trivial. Let $n \geq 1$ and assume the claim is true for $n-1$. By induction hypothesis $L \neq \bigcup_{1 \leq i \leq n,\, i \neq j} L_i$ for all indices $j$. Thus, \cref{groupintersect} applied to the additive groups implies that the intersection $K \coloneqq L_1 \cap \cdots \cap L_n$ is infinite. Now we may regard $L$ as a vector space over $K$ and $L_i$ as subspaces of $L$. Then \cref{vecunion} gives the desired contradiction.
\end{proof}

\begin{rem}
For what follows, actually a weaker form of \cref{fieldunion} is sufficient, namely that for a finite extension $L/K$ the field $L$ is not the union of finitely many proper intermediate fields. But this follows immediately from \cref{vecunion} if $K$ is infinite, and for finite fields $K$ the field $L$ is also finite, so that $L^{\times}$ is cyclic and we are done. This shortens the proof, but we did not choose this path here since it would restrict \cref{subgroupclass} below to finite extensions.
\end{rem}

% ----------------------- CLASSIFICATION OF SUBGROUPS ----------------------- 

\section{Classification of subgroups}

\begin{defi}
Let $L/K$ be an extension of fields and $H \subseteq \Aut_K(L)$ be a subgroup. We define the \emph{fixed field} as
\[L^H \coloneqq \{a \in L : \forall \sigma \in H ~ (\sigma(a)=a) \}.\]
\end{defi}

This is clearly an intermediate field of $L/K$.

We can rephrase the definition using group actions: In fact, the group $\Aut_K(L)$ acts on $L$ in a natural way, and $L^H$ is nothing but the field of fixed points of this action when restricted to the subgroup $H$.

We have the obvious relationship $H \subseteq \Aut_{L^H}(L)$. In some cases, the converse is also true:

\begin{prop} \label{subgroupclass}
Let $L/K$ be an extension and $H \subseteq \Aut_K(L)$ be a finite subgroup. Then
\[H = \Aut_{L^H}(L).\]
\end{prop}

\begin{proof}
Let $\tau \in \Aut_{L^H}(L)$. We need to prove $\tau \in H$. First, we claim that
\[L = \bigcup_{\sigma \in H} \{\sigma = \tau\},\]
where $\{\sigma = \tau\}$ is a short notation for the subfield $\{a \in L : \sigma(a) = \tau(a)\}$, the equalizer of $\sigma,\tau$. Let $a \in L$. In order to use the assumption that $\tau$ fixes elements of $L^H$, we need to somehow come up with elements of $L^H$. Consider the polynomial
\[p \coloneqq \prod_{\sigma \in H} \bigl(T - \sigma(a)\bigr) \in L[T].\]
Of course, this is only well-defined since $H$ is finite, and we have $p(a)=0$. The natural action of $H$ on $L$ extends to an action on $L[T]$. The polynomial $p$ is clearly fixed by this action since the action just permutes the linear factors. Hence, we have
\[p \in L[T]^H = L^H[T].\]
Thus, $\tilde{\tau} : L[T] \to L[T]$ fixes $p$, i.e.\ $\tilde{\tau}(p)=p$. Since $\tau(a)$ is a root of $\tilde{\tau}(p)=p$, there is some $\sigma \in H$ with $\sigma(a)=\tau(a)$, so that $a \in \{\sigma = \tau\}$. This proves our claim.

Now, each equalizer $\{\sigma = \tau\}$ is a subfield of $L$. Thus, \cref{fieldunion} implies that there is some $\sigma \in H$ with $L = \{\sigma = \tau\}$, which just means $\sigma = \tau$.
\end{proof}

% ----------------------- SEPARABLE EXTENSIONS ----------------------- 

\section{Separable extensions}

\begin{defi}
Let $\alg{K}$ be an algebraic closure of $K$ and $L/K$ be an algebraic extension. We call $a \in L$ \emph{separable} over $K$ if its minimal polynomial $f \in K[T]$ has only simple roots in $\alg{K}$. The extension $L/K$ is called \emph{separable} if every element of $L$ is separable over $K$.
\end{defi}

\begin{lemma} \label{sepsub}
Let $L/E/K$ be two algebraic extensions. If $L/K$ is separable, then $L/E$ and $E/K$ are separable as well.
\end{lemma}

\begin{proof}
The claim for $E/K$ is trivial. The claim for $L/E$ follows from the observation that the minimal polynomial of an element of $L$ over $E$ divides the minimal polynomial over $K$.
\end{proof}

\begin{prop} \label{counthom}
If $L/K$ is finite separable, then $\Hom_K(L,\alg{K})$ has exactly $[L:K]$ elements.
\end{prop}

\begin{proof}
We can just recycle the proof of \cref{atmost} (which showed inequality). In the induction step we only need to observe 1) that $K \to \alg{K}$ admits exactly $[K(a):K]$ extensions to $K(a)$ because $a$ is separable, and 2) that $L/K(a)$ is separable by \cref{sepsub}.
\end{proof}

\begin{rem}
Actually, a finite extension $L/K$ is separable if and only if $\Hom_K(L,\alg{K})$ has $[L:K]$ elements \cite[Lem.\ VII.4.24]{Aluffi}. This is crucial for the theory of separable extensions.
\end{rem}

\begin{lemma} \label{ground}
Let $L/K$ be a separable extension and $a \in L$ be an element. Assume that for all $K$-homomorphisms $\sigma,\tau : L \to \alg{K}$ we have $\sigma(a)=\tau(a)$. Then $a \in K$.
\end{lemma}

\begin{proof}
Because every $K$-homomorphism $K(a) \to \alg{K}$ extends to $L$ by \cref{extend}, we may assume $L = K(a)$. Because of \cref{univ} the assumption means that the minimal polynomial of $a$ has exactly one root in $\alg{K}$. On the other hand, it has only simple roots, since $a$ is separable. Hence, it must be a linear polynomial, meaning $a \in K$.
\end{proof}

% ----------------------- NORMAL EXTENSIONS ----------------------- 

\section{Normal extensions}

Recall that a right action of a group $G$ on a set $X$ is called \emph{transitive} if $X$ is non-empty and for all $x,y \in X$ there is some $g \in G$ with $y = xg$. Equivalently, $X$ has exactly one $G$-orbit.

\begin{defi}\label{defnormal}
An algebraic extension $L/K$ is called \emph{normal} if the natural right action of $\Aut_K(L)$ on the set $\Hom_K(L,\alg{K})$ is transitive.
 
Note that the set $\Hom_K(L,\alg{K})$ is non-empty by \cref{extend}. So the definition means that for all $K$-homomorphisms $\sigma,\tau : L \to \alg{K}$ there is some $K$-automorphism $\varphi : L \to L$ with $\tau = \sigma \circ \varphi$. Here, $\varphi$ is unique since $\sigma$ is injective.

Thus, if we fix a $K$-homomorphism $\sigma : L \to \alg{K}$, the extension $L/K$ is normal if and only if the map
\[\Aut_K(L) \stackrel{\ref{auto}}{=\joinrel=} \Hom_K(L,L) \to \Hom_K(L,\alg{K}),\, \varphi \mapsto \sigma \circ \varphi\]
is bijective. Also, for the existence of $\varphi$ above it is clearly enough to check $\im(\tau) \subseteq \im(\sigma)$. This observation together with $\alg{E}=\alg{K}$ already implies the next result.
\end{defi}

\begin{lemma}\label{normsub}
Let $L/E/K$ be two algebraic extensions. If $L/K$ is normal, then $L/E$ is normal as well. \hfill \qed
\end{lemma}

\begin{prop} \label{fix}
Let $L/K$ be a normal separable extension and $E$ be an intermediate field of $L/K$. Then
\[E = L^{\Aut_E(L)}.\]
\end{prop}

\begin{proof}
By \cref{sepsub,normsub} the extension $L/E$ is normal and separable, so that it suffices to treat the special case $E=K$. Let $a \in L^{\Aut_K(L)}$. For all $K$-homomorphisms $\sigma,\tau : L \to \alg{K}$ there is some $\varphi \in \Aut_K(L)$ with $\tau = \sigma \circ \varphi$. We get $\tau(a)= \sigma(\varphi(a)) = \sigma(a)$. Thus, \cref{ground} implies $a \in K$.
\end{proof}

For the sake of completeness, we include the equivalence between \cref{defnormal} and a more common definition of a normal extension.

\begin{lemma} \label{normalchar}
An algebraic extension $L/K$ is normal if and only if every irreducible polynomial $f \in K[T]$ which has a root in $L$ splits completely over $L$, i.e.\ is a product of linear factors.
\end{lemma}

\begin{proof}
Choose some $K$-homomorphism $\sigma : L \to \alg{K}$. The splitting property means that for every $a \in L$ its minimal polynomial splits completely over $L$. Equivalently, its roots in $\alg{K}$ are all contained in $\im(\sigma)$. By \cref{univ} and \cref{extend} these roots are $\tau(a)$ for $\tau \in \Hom_K(L,\alg{K})$. So the condition is just $\im(\tau) \subseteq \im(\sigma)$ for all $\tau \in \Hom_K(L,\alg{K})$.
\end{proof}

% ----------------------- FUNDAMENTAL THEOREM ----------------------- 

\section{Fundamental theorem of Galois theory}

We are now able to combine the results from the previous sections.

\begin{defi}
A \emph{Galois extension} is a normal separable algebraic extension.
\end{defi}

\begin{rem} \label{galoissup}
Notice that for a Galois extension $L/K$ and an intermediate field $E$ the extension $L/E$ is also Galois by \cref{sepsub,normsub}.
\end{rem}

\begin{thm}\label{order}
Let $L/K$ be a finite Galois extension. Then $\Aut_K(L)$ is a finite group of order $[L:K]$, called the \emph{Galois group} of $L/K$.
\end{thm}

\begin{proof}
Since $L/K$ is normal, we have
\[\Aut_K(L) \cong \Hom_K(L,\alg{K}),\]
and since $L/K$ is separable, $\Hom_K(L,\alg{K})$ has exactly $[L:K]$ elements by \cref{counthom}.
\end{proof}

Let us briefly mention that the proof of \cref{counthom} can actually be used to compute Galois groups in examples.

We are now ready to prove the main theorem.

\begin{thm}[Fundamental theorem of Galois theory]\label{fundamental}
Let $L/K$ be a finite Galois extension.
\begin{enumerate}
\item\label{correspondence} The maps $E \mapsto \Aut_E(L)$ and $L^H \mapsfrom H$ are inverse to each other and hence establish a bijection
\[\bigl\{\text{intermediate fields of } L/K\bigr\} \cong \bigl\{\text{subgroups of } \Aut_K(L)\bigr\}.\]
\item These maps are inclusion-reversing in the sense
\begin{itemize}
\item $E \subseteq E' \implies \Aut_{E'}(L) \subseteq \Aut_E(L)$
\item $H \subseteq H' \implies L^{H'} \subseteq L^H$
\end{itemize}
and hence provide an anti-isomorphism of partial orders.
\item The degree of an intermediate field $E$ is the index of the corresponding subgroup:
\[[E:K] = [\Aut_K(L) : \Aut_E(L)]\]
\item For intermediate fields $E,E'$ and subgroups $H,H'$ the following relationships hold:
\begin{itemize}
\item $\Aut_{E \cap E'}(L) = \langle \Aut_E(L),\Aut_{E'}(L) \rangle$
\item $\Aut_{E \cdot E'}(L) = \Aut_E(L) \cap \Aut_{E'}(L)$
\item $L^{H \cap H'} = L^H \cdot L^{H'}$
\item $L^{\langle H,H' \rangle} = L^H \cap L^{H'}$
\end{itemize}
Here, $E \cdot E'$ denotes the compositum of $E$ and $E'$.
\item For an intermediate field $E$ the extension $E/K$ is normal (and hence a Galois extension) if and only if $\Aut_E(L)$ is a normal subgroup of $\Aut_K(L)$. In this case, we have
\[\Aut_K(L) / \Aut_E(L) \cong \Aut_K(E).\]
\item The bijection from (1) restricts to a bijection
\[\bigl\{\text{normal intermediate fields of } L/K\bigr\} \cong \bigl\{\text{normal subgroups of } \Aut_K(L)\bigr\}.\]
\end{enumerate}
\end{thm}

\begin{proof}
(1) For an intermediate field $E$ of $L/K$ we have $E = L^{\Aut_E(L)}$ by \cref{fix}. For a subgroup $H$ of $\Aut_K(L)$ we have $H = \Aut_{L^H}(L)$ by \cref{subgroupclass}, which is applicable since $H$ is finite by \cref{atmost}.

(2) The verification of these inclusions is trivial.

(3) By \cref{galoissup,order} we have
\[[\Aut_K(L):\Aut_E(L)] = \ord(\Aut_K(L)) / \ord(\Aut_E(L)) = [L:K] / [L:E] = [E:K].\]

(4) This follows from (2) as follows: The usual definition of a supremum as the least upper bound works in every partial order. Similarly for an infimum.
In the partial order of subgroups of a group we have $\sup(H,H') = \langle H,H' \rangle$ and $\inf(H,H') = H \cap H'$.
In the partial order of intermediate fields of an extension we have $\sup(E,E') = E \cdot E'$ as well as $\inf(E,E') = E \cap E'$.
Now we may use the general and easy fact that an anti-isomorphism of partial orders transforms suprema into infima and vice versa.

(5) Let $E/K$ be normal. Fix a $K$-homomorphism $L \to \alg{K}$. Because of \cref{extend} the restriction map $\Hom_K(L,\alg{K}) \to \Hom_K(E,\alg{K})$ is surjective, and since $L/K$ and $E/K$ are normal, it identifies with a restriction map $\Aut_K(L) \to \Aut_K(E)$. This is clearly a homomorphism of groups whose kernel is $\Aut_E(L)$. Thus, $\Aut_E(L)$ is a normal subgroup with $\Aut_K(L) / \Aut_E(L) \cong \Aut_K(E)$.

For the other direction assume that $H$ is a normal subgroup of $\Aut_K(L)$. To prove that $L^H$ is normal over $K$, choose two $K$-homomorphisms $\sigma,\tau : L^H \to \alg{K}$. By \cref{extend} there are extensions $\sigma',\tau'$ to $L$. Since $L/K$ is normal, we have $\sigma' = \tau' \circ \varphi$ for some $\varphi : L \to L$. We claim $\varphi(L^H) \subseteq L^H$. In fact, for every $a \in L^H$ and $\psi \in H$ we have $\varphi^{-1} \psi \varphi \in H$ (since $H$ is normal), hence $(\varphi^{-1} \psi \varphi)(a)=a$, i.e.\ $\psi(\varphi(a)) = \varphi(a)$. Now, from $\varphi(L^H) \subseteq L^H$ we deduce
\[\sigma(L^H) = \sigma'(L^H) = \tau'(\varphi(L^H)) \subseteq \tau'(L^H) = \tau(L^H).\]

(6) This follows from (1) and (5).
\end{proof}

\begin{rem}
Both \cref{order} and \cref{fundamental}(\ref{correspondence}) actually characterize Galois extensions by \cite[Thm.\ VII.6.9]{Aluffi}.
\end{rem}

% ----------------------- REFERENCES ----------------------- 

\end{document}